\documentclass{article}
\usepackage[utf8]{inputenc}
\usepackage{amsfonts}
\usepackage{amsthm}
\usepackage{amsmath}
\usepackage{biblatex}
\title{Measures Determined by the Restriction of Convolution Powers to the Proper Concave Cone}
\author{Aleksander Pawlewicz}
\date{February 2022}

\newtheorem{theorem}{Theorem}
\newtheorem{lemma}[theorem]{Lemma}
\newtheorem{remark}{Remark}

\newcommand{\supp}{\mathrm{supp}\,}
\newcommand{\Supp}{\mathrm{supp_C}\,}
\newcommand{\conv}{\mathrm{conv}\,}
\begin{document}

\maketitle

\begin{abstract}
    Let $\mu$ and $\nu$ be two non-degenerate finite signed Borel measures defined on a proper convex cone of $\mathbb{R}^n$. We prove that if all convolution powers of $\mu$ and $\nu$ are appropriately equal (and non-zero) on a proper concave cone of $\mathbb{R}^n$, the measures are equal. A similar but more general result for measures defined on $\mathbb{R}$ can be found in \cite{Kwasnicki}. We also provide an example of two-dimensional measures, which indicates that equality of measures and their appropriate convolution powers on a half-plane is not enough for equality of measures.
\end{abstract}

\section{Introduction}

In this article, we study some conditions connected with the convolution powers of finite signed Borel measures which uniquely determine a measure. The inspiration to write this article was paper \cite{Kwasnicki}, in which, a stronger one-dimensional version of Theorem \ref{Maintheorem} was presented. Also, a brief description of the history of the problem can be found there. There are a few other important articles devoted to matters investigated in our paper. We will mention them later. 

The key question is: 
\newline

How much should we know of a measure in order to determine this measure uniquely? 
\newline

However, the above question is to general to be considered.  Hence, we not only assume that we know the behaviour of a measure on some part of the space, but we also know the behaviour of all the convolution powers of the measure in this region.

The question of A. N. Kolmogorov from the 1950s (see the beginning of section 1 in \cite{OsUl}) can be considered to be the starting point of systematic research in this direction. Paper \cite{OsUl} contains, in fact, a general overview of the matters connected with the unique determination of probability distributions and Borel measures in the spirit of our article. There is also formulated not true theorem similar to Theorem \ref{Maintheorem} (see Theorem 4 in \cite{OsUl}), but without proof, which state that it is enough to assume the equality of measures and appropriate convolution powers of measures on a half-space in order to obtain equality of the measures. We will not give a detailed description of the literature of the subject. We mention only an interesting but not very recent book \cite{analytic} on general analytic methods in probability theory. Also, the literature mentioned in \cite{Kwasnicki} could be recommended. 

Although paper \cite{Kwasnicki} relies on analytic methods, we try to deal directly with measures and a convolution operator. We analyse consecutive convolution powers of measures and their supports. This line of investigation of similar problems was previously used in paper \cite{Ost}. 

We would also like to recall the remarkable Titchmarsh convolution theorem first formulated in paper \cite{Tit}. This theorem is the key 'ingredient' in our considerations. There is extensive literature devoted to this theorem and its different proof methods, see \cite{Gar} and references mentioned there. One of its general forms can be stated as follows (see \cite{Wei}):

\begin{theorem}
\label{Titchmarsh1}
Let $\mu_1$ and $\mu_2$ be finite with compact support measures on locally compact abelian group $G$. If $G$ has no compact subgroups then $$\conv\left(\supp(\mu_1*\mu_2)\right)=\conv\left(\supp(\mu_1)\right)+\conv\left(\supp(\mu_2)\right),$$
where $\conv(A)$ is the convex closure of a set $A$.
\end{theorem}

In paper \cite{Wei} just after the proof of Theorem \ref{Titchmarsh1} the author indicated that the stronger version of Titchmarsh theorem is true with the minor changes of the proof of Theorem \ref{Titchmarsh1}. Actually, we will need this stronger version.

\begin{theorem}
\label{Titchmarsh2}
Let $a$ and $b$ be two non-degenerate finite signed Borel measures with supports contained in a strictly convex cones of $\mathbb{R}^n$ such that
$$\Supp a=k \mbox{ and } \Supp b=l.$$
Then we have
$$\Supp (b*a)=k+l.$$
\end{theorem}
For definition of $\Supp$ see below.
Finally, we would just like to mention that there are connections between the one-dimensional version of measure theory problems described in this paper and the theory of random walks, see \cite{Kwasnicki}.

Let us make one important technical assumption.

\begin{remark}
All sets will be measurable with respect to the measures considered in a particular place.
\end{remark}

Now, let us introduce notation. For $p\in\mathbb{R}$ we will use the symbol $\mathbb{R}_p^n$ to denote the left half-space,
$$\mathbb{R}_p^n=\left\{x=(x_1, x_2, ..., x_n)\in\mathbb{R}^n|\, x_1\leq p\right\},$$
and also
$$\mathbb{R}_+^n=\left\{x=(x_1, x_2, ..., x_n)\in\mathbb{R}^n|\, x_1>0\right\},$$

We will denote by $\mathcal{C}$ a collection of the strictly convex closed $n$-dimensional cones with vertex at $0$, symmetric with the $Ox_1$ axe and such that $C\cap\mathbb{R}_+^n=\emptyset$ for $C\in\mathcal{C}$.
Also, we define
$$C(p)=C+\{(p,0, ...,0)\},$$
for some $C\in\mathcal{C}$ and $p\in\mathbb{R}$, where $C+\{(p,0, ...,0)\}$ means the Minkowski sum of the sets $C$ and $\{(p,0, ...,0)\}$.

The support of the finite signed Borel measure $\mu$ on $\mathbb{R}^n$, $\mbox{supp}\,\mu$, is the smallest closed set $A$ such that the total variation of measure $\mu$ fulfil the condition
$$|\mu|(\mathbb{R}^n\setminus A)=0.$$

For the non-degenerate finite signed Borel measure $\mu$, cone $C\in\mathcal{C}$ and a real number $p$, we will denote $$\Supp\mu=p\,\,\iff\,\,p=\inf\left\{t\in\mathbb{R}|\, \supp \mu\subset C(t)\right\}.$$
Thus, the equality $\Supp\mu=p$ means that there is some amount of measure $\mu$ concentrated arbitrarily close to the conical surface of the cone $C(p)$ inside the cone $C(p)$. Because the support of a measure $\mu$ and sets $C(t)$ are closed, the above infimum is attained.

We define the measure $\mu*\nu$ --- the convolution of finite signed Borel measures $\mu$ and $\nu$ --- by the standard formula
$$\mu*\nu(A)=\int_{\mathbb{R}^n}\mu(A-y)\,d\nu(y),$$
for every measurable set $A\subset\mathbb{R}^n$.
The measure $\mu*\nu$ obtained in this way is also a finite signed Borel measure and the convolution operation, $(\mu,\nu)\mapsto\mu*\nu$, is associative and commutative. 

We will use the following convention for the finite signed Borel measure $\mu$:
$$\mu^{*1}=\mu \mbox{ and }\mu^{*(k+1)}=\mu*\mu^{*k}.$$

\section{The main result}

In this article, we prove the following statement.

\begin{theorem}
\label{Maintheorem}
Let $C \in \mathcal{C}$ be a strictly convex cone and let $\mu$ and $\nu$ be finite signed Borel measures supported in $C(h)$ for some $h>0$.
Moreover, let us assume that for every $k=1, 2, 3, ...$ we have
\begin{equation}
\label{rownanie1}
\mu^{*_k}(A)=\nu^{*_k}(A),
\end{equation}
for every measurable set $A\subset\mathbb{R}^n\setminus C$. Also, restrictions of the measures $\mu$ and $\nu$ to the set $\mathbb{R}^n\setminus C$ are non-zero measures.
Then $\mu=\nu$.
\end{theorem}

Now, we give an example which shows why it is not enough to assume the equality of successive convolution powers on the half-plane $\mathbb{R}_+^2$. Let us denote by $F(x)$ the Fej\'er kernel,
$$F(x)=\frac{1}{2\pi}\left(\frac{\sin{x/2}}{x/2}\right)^2=\frac{1}{2\pi}\int_{-1}^1(1-|\xi|)e^{i\xi x}\,d\xi.$$
We define the measures in the following way:
$$d\mu(x,y)=F(y)d\delta_1(x)dy +F(y)\left(e^{2iy}+e^{-2iy}\right)d\delta_{-3}(x)dy$$
and
$$d\nu(x,y)=F(y)d\delta_1(x)dy +F(y)\left(e^{10iy}+e^{-10iy}\right)d\delta_{-2}(x)dy,$$
where the symbol $\delta_p$ means the Dirac delta point mass at a point $p$.
Notice that
$$\widehat{\mu^{*k}}(s,t)=
\begin{cases}
e^{-iks}\left(1-|t|\right)^k, & \text{if } t\in(-1,1),\\
e^{3iks}\left(1-|t+2|\right)^k, & \text{if } t\in(-3,-1),\\
e^{3iks}\left(1-|t-2|\right)^k, & \text{if } t\in(1,3),\\
0, & \text{otherwise},
\end{cases}$$
and
$$\widehat{\nu^{*k}}(s,t)=
\begin{cases}
e^{-iks}\left(1-|t|\right)^k, & \text{if } t\in(-1,1),\\
e^{2iks}\left(1-|t+10|\right)^k, & \text{if } t\in(-11,-9),\\
e^{2iks}\left(1-|t-10|\right)^k, & \text{if } t\in(9,11),\\
0, & \text{otherwise},
\end{cases}$$
for $k\in\mathbb{N}$. Fourier transform uniquely determines the measure, thus we have
$\mu^{*k}\vert_{\mathbb{R}_+^2}=\nu^{*k}\vert_{\mathbb{R}_+^2}$ but $\mu\neq\nu$.

In the next paragraph, we will present proof of Theorem \ref{Maintheorem} but, before that, let us formulate two auxiliary lemmas. They describe the behaviour of the supports of signed measures under the convolution operation.

\begin{lemma}
\label{lemat1}
Let $a$ and $b$ be two finite signed Borel measures on $\mathbb{R}^n$ and $C\in\mathcal{C}$. Suppose that $\supp a\subset C(h)$ and $\supp b\subset C(h)$, for some $h>0$. Assume also that
$$\Supp (b*a)\leq 0,$$
and
$$\Supp a=-p$$
for real number $p\geq 0$.
Then $\Supp b\leq p$ (or equivalently $\supp b\subset C(p)$).
\end{lemma}

\begin{proof}
Assume, by contradiction, that $\supp b\not\subset\mathbb C(p)$. Then $\Supp b=p+c$, for some $c>0$. Now, let us consider two cases. If $\Supp a=-p$, then by Titchmarsh theorem for measures supported in a strictly convex cone (Theorem \ref{Titchmarsh2}) we get $\Supp (b*a)>0$. We have a contradiction. On the other hand, if $\supp(b*a)\subset\mathbb{R}_0^n$, then $\Supp a=-p-c$. One more time we have a contradiction.
\end{proof}

The second required lemma is devoted to finding information about the support of the measures of a specific long-sum form.

\begin{lemma}
\label{lemat2}
Let $a$ and $b$ be two non-degenerate finite signed Borel measures such that
$$\Supp a=r \mbox{ and } \Supp b=r.$$
for some positive number $r$ and $C\in\mathcal{C}$. Let us also assume that for every $k=1, 2, 3, ...$ we have
\begin{equation}
\label{rownanie7}
a^{*_k}(A)=b^{*_k}(A),
\end{equation}
for every measurable set $A\subset\mathbb{R}^n\setminus C$, and restrictions of measures $a$ and $b$ to the set $\mathbb{R}^n\setminus C$ are non-zero measures.
Then for every natural number $k$ we have
\begin{equation}
\label{wsm}
\Supp\left(a^{*k}+a^{*(k-1)}*b+a^{*(k-2)}*b^{*2}+...+a*b^{*(k-1)}+b^{*k}\right)=k\cdot r.
\end{equation}
\end{lemma}

\begin{proof}
We need to estimate the value
$$\Supp\left(a^{*k}+a^{*(k-1)}*b+a^{*(k-2)}*b^{*2}+...+a*b^{*(k-1)}+b^{*k}\right).$$
For $k=1$ measures $a$ and $b$ are equal on the set $\mathbb{R}^n\setminus C$ by assumption; thus, we have
\begin{equation*}
\Supp(a+b)=\Supp(2a)=r.
\end{equation*}
By the same argument we get equality 
\begin{equation}
\label{sumar} 
\Supp (ra+b)=\Supp a,
\end{equation}
for natural $r$.

For $k=2$ we have, by Theorem \ref{Titchmarsh2} and above equality \eqref{sumar},
\begin{equation*}
2\cdot r
=\Supp\left(a*a\right)
=\Supp\left(a*\left(2a+b\right)\right).
\end{equation*}
Moreover,
\begin{equation*}
\begin{split}
\Supp\left(a*\left(2a+b\right)\right)
&=\Supp\left(2a^{*2}+a*b\right) \\
&=\Supp\left(a^{*2}+a*b+b^{*2}\right).
\end{split}
\end{equation*}
Above, we used assumption \eqref{rownanie7}, that is the equality of measures $a^{*2}$ and $b^{*2}$ on the set $C(2r)\setminus C$.

In the same way, we can prove equation \eqref{wsm} for every natural number $k$.
Let $k$ be a natural number. Then, by Theorem \ref{Titchmarsh2} and equality \eqref{sumar}, we have 
\begin{equation*}
k\cdot r
=\Supp\left(a^{*(k-1)}*a\right)
=\Supp\left(a^{*(k-1)}*\left(ka+b\right)\right).
\end{equation*}
Furthermore,
\begin{equation*}
\begin{split}
\Supp\left(a^{*(k-1)}*\left(ka+b\right)\right)
&=\Supp\left(a^{*(k-2)}*\left(ka^{*2}+a*b\right)\right) \\
&=\Supp\left(a^{*(k-2)}*\left((k-1)a^{*2}+a*b+b^{*2}\right)\right).
\end{split}
\end{equation*}
Above, we used assumption \eqref{rownanie7}, that is the equality of measures $a^{*2}$ and $b^{*2}$ on the set $C(2r)\setminus C$. Changing the above expression step by step in the same way and using the equality of successive convolution powers of measures $a$ and $b$ (assumption \eqref{rownanie7}), we finally get
\begin{equation*}
\begin{split}
\Supp&\left(a^{*(k-1)}*\left(ka+b\right)\right) \\
&=... \\
&=\Supp\left(a^{*2}*\left(4a^{*(k-2)}+a^{*(k-3)}*b+...+a*b^{*(k-3)}\right)\right) \\
&=\Supp\left(a^{*2}*\left(3a^{*(k-2)}+a^{*(k-3)}*b+...+a*b^{*(k-3)}+b^{*(k-2)}\right)\right) \\
&=\Supp\left(a*\left(3a^{*(k-1)}+a^{*(k-2)}*b+...+a*b^{*(k-2)}\right)\right) \\
&=\Supp\left(a*\left(2a^{*(k-1)}+a^{*(k-2)}*b+...+a*b^{*(k-2)}+b^{*(k-1)}\right)\right) \\
&=\Supp\left(2a^{*k}+a^{*(k-1)}*b+...+a^{*2}*b^{*(k-2)}+a*b^{*(k-1)}\right) \\
&=\Supp\left(a^{*k}+a^{*(k-1)}*b+...+a^{*2}*b^{*(k-2)}+a*b^{*(k-1)}+b^{*k}\right).
\end{split}
\end{equation*}
\end{proof}

Now, we can turn to the proof of Theorem \ref{Maintheorem}.

\section{Proof of the main theorem}

With all facts concerned in Section 2 we can deal with the most important statement of this paper.

\begin{proof}[Proof of Theorem \ref{Maintheorem}]

Let $C\in\mathcal{C}$ and let $\mu$ and $\nu$ be two signed Borel measures described by the assumptions of the theorem. Let us also assume that $\mu\neq\nu$.

We know that
$$\mu*\mu(A)=\nu*\nu(A)$$
for every measurable set $A\subset\mathbb{R}^n\setminus C$. Therefore, for such a set $A$ we have
\begin{equation}
\label{skroconemnozenie}
0=\left(\mu*\mu-\nu*\nu\right)(A)=\left((\mu+\nu)*(\mu-\nu)\right)(A),
\end{equation}
by commutativity of convolution. Moreover,
$$\supp (\mu-\nu), \supp (\mu*\mu-\nu*\nu) \subset C \subset\mathbb{R}_0^n,$$
by equality \eqref{rownanie1} for $k=1$ and for $k=2$.

Let
$$-p=\Supp (\mu-\nu),$$
for some $p\geq0$.
We will consider two cases: $p=0$ and $p>0$.\newline

In the case $p=0$, by Lemma \ref{lemat1} (take $a=\mu-\nu$ and $b=\mu+\nu$) and equation \eqref{skroconemnozenie}, we get
$$\supp (\mu+\nu)\subset C.$$
This could not occur because
$$\supp (2\mu)=\supp\left((\mu+\nu)+(\mu-\nu)\right)\subset C,$$
and we assumed that the restriction of the measure $\mu$ to the set $\mathbb{R}^n\setminus C$ is a non-zero measure.\newline

In the case $p>0$, by Lemma \ref{lemat1} ($a$ and $b$ as above), equation \eqref{skroconemnozenie} and assumption \eqref{rownanie1}, we have
$$\Supp \mu=r \,\mbox{ and }\, \Supp\nu=r,$$
for some $0<r<p$, and also
$$\mu=\nu \mbox{ on the set } 
C(r)\setminus C(-p).$$
Moreover, by Theorem \ref{Titchmarsh2} we have
$$\Supp \mu^{*k}=k\cdot r \,\mbox{ and }\,  \Supp \nu^{*k}= k\cdot r,$$
for $k= 1, 2, 3, ...$ .

Now, by assumption we know that 
$$\left(\mu^{*k}-\nu^{*k}\right)(A)=
\left(\mu-\nu\right)
*
\left(\mu^{*(k-1)}+\mu^{*(k-2)}\nu+...+\nu^{*(k-1)}\right)(A)=0,$$
for every set $A\subset\mathbb{R}^n\setminus C$ and natural number $k$. 
Notice that 
$$\Supp (\mu-\nu)=-p,$$
and, by Lemma \ref{lemat2}, for every natural number $k$ we have
$$\Supp \left(\mu^{*(k-1)}+\mu^{*(k-2)}\nu+\mu^{*(k-3)}\nu^{*2}+...+\nu^{*(k-1)}\right)
=(k-1)\cdot r.$$ Thus, for a sufficiently big natural number $k$ we get
$$\Supp
\left[\left(\mu-\nu\right)
*
\left(\mu^{*(k-1)}+\mu^{*(k-2)}\nu+\mu^{*(k-3)}\nu^{*2}+...+\nu^{*(k-1)}\right)\right]>0,$$
by Theorem \ref{Titchmarsh2}. However, this means that $\Supp (\mu^{*n}-\nu^{*n})>0$.
We have a contradiction.
\end{proof}

Despite all the attempts, the author could not drop the assumption saying that two measures $\mu$ and $\nu$ are supported in a common cone $C(h)$, $h>0$, from Theorem \ref{Maintheorem}. The question is whether this assumption is, indeed, required.

\section*{Acknowledgement}
The author would like to express his gratitude to Rafał Latała and Krzysztof Oleszkiewicz for interesting discussion during the seminar of probability theory which shed new light on the shape of the article. I am also grateful to Mateusz Wasilewski for his numerous comments on the early version of this article.

\printbibliography

\end{document}